\documentclass[12pt,reqno]{article}

\usepackage[square, comma, sort&compress, numbers]{natbib}
\usepackage{amsmath,amsthm,amsfonts,amssymb,latexsym,mathrsfs,color}

\newtheorem{thm}{Theorem}[section]
\newtheorem{lem}[thm]{Lemma}

\newtheorem{prop}[thm]{Proposition}

\theoremstyle{definition}

\newtheorem{rem}[thm]{Remark}

\numberwithin{equation}{section}

\def\la{\lambda}

\title{Total positivity of Riordan arrays
}
\author{Xi Chen, Huyile Liang and Yi Wang
\thanks{{\it Email address:}\quad wangyi@dlut.edu.cn (Y. Wang)}}
\date{\footnotesize School of Mathematical Sciences, Dalian University of Technology, Dalian 116024, PR China}

\begin{document}

\maketitle

\begin{abstract}
%Riordan arrays play an important unifying role in enumerative combinatorics.
We present sufficient conditions for total positivity of Riordan arrays.
As applications we show that many well-known combinatorial triangles are totally positive
and many famous combinatorial numbers are log-convex in a unified approach.
%\bigskip
\\[1pt]
\end{abstract}
%%%%%%%%%%%%%%%%%%%%%%%%%%%%%%%%%%%%%%%%%%%
\section{Introduction}
\hspace*{\parindent}
%%%%%%%%%%%%%%%%%%%%%%%%%%%%%%%%%%%%%%%%%%%
Riordan arrays play an important unifying role in enumerative combinatorics
\cite{SGWW91}.
A {\it (proper) Riordan array}, denoted by $(g(x),f(x))$, is an infinite lower triangular matrix
whose generating function of the $k$th column is $x^kf^k(x)g(x)$ for $k=0,1,2,\ldots$,
where $g(0)=1$ and $f(0)\neq 0$.
A Riordan array $R=[r_{n,k}]_{n,k\ge 0}$
can also be characterized by two sequences
$(a_n)_{n\ge 0}$ and $(z_n)_{n\ge 0}$ such that
\begin{equation}\label{rrr-c}
r_{0,0}=1,\quad r_{n+1,0}=\sum_{j\ge 0}z_jr_{n,j},\quad r_{n+1,k+1}=\sum_{j\ge 0}a_jr_{n,k+j}
\end{equation}
for $n,k\ge 0$ (see \cite{CKS12,HS09,MRSV97,Spr94} for instance).
Call $(a_n)_{n\ge 0}$ and $(z_n)_{n\ge 0}$
the {\it $A$}- and {\it $Z$-sequences} of $R$ respectively.
%Let $A(x)$ and $Z(x)$ be the generating functions of $A$- and $Z$-sequences respectively.

Many triangles in combinatorics are Riordan arrays with simple $A$- and $Z$-sequences.
For example,
the Pascal triangle with $Z=(1,0,\ldots)$ and $A=(1,1,0,\ldots)$,
the Catalan triangle with $Z=(2,1,0\ldots)$ and $A=(1,2,1,0,\ldots)$~\cite{Sha76},
the Motzkin triangle with $Z=(1,1,0,\ldots)$ and $A=(1,1,1,0,\ldots)$~\cite{Aig99},
the ballot table with $Z=A=(1,1,1,\ldots)$\cite{Aig01},
the large Schr\"oder triangle with $Z=(2,2,2,\ldots)$ and $A=(1,2,2,\ldots)$~\cite{CKS12},
and the little Schr\"oder triangle with $Z=A=(1,2,2,\ldots)$~\cite{CKS12}.
%See \cite{Aig01,CKS12,HS09,Sha76} and Sloane's OEIS~\cite{Slo} for details.
Such triangles arise often in the enumeration of lattice paths,
e.g., the Dyck paths, the Motzkin paths, and the Schr\"{o}der paths and so on \cite{CKS12,MRSV97,SGWW91}.
The $0$th column of such an array counts the corresponding lattice paths,
including the Catalan numbers, the Motzkin numbers, the large and little Schr\"{o}der numbers.
%Riordan arrays have many fascinating combinatorial properties and
There have been quite a few papers concerned with combinatorics of Riordan arrays
(see \cite{CKS12,HS09,MRSV97,Rog78,SGWW91,Spr94} for instance).
Our concern in the present paper is total positivity of Riordan arrays.

Following Karlin~\cite{Kar68},
an infinite matrix is called {\it totally positive of order $r$} (or shortly, {\it TP$_r$}),
if its minors of all orders $\le r$ are nonnegative.
The matrix is called {\it TP} if its minors of all orders are nonnegative.
Let $(a_n)_{n\ge 0}$ be an infinite sequence of nonnegative numbers.
It is called a {\it P\'olya frequency sequence of order $r$} (or shortly, a {\it PF$_r$} sequence),
if its Toeplitz matrix
$$[a_{i-j}]_{i,j\ge 0}=\left[
\begin{array}{lllll}
a_{0} &  &  &  &\\
a_{1} & a_{0} & &\\
a_{2} & a_{1} & a_{0} &  &\\
a_{3} & a_{2} & a_{1} & a_{0} &\\
\vdots & &\cdots & & \ddots\\
\end{array}
\right]$$
is TP$_r$.
It is called {\it PF} if its Toeplitz matrix is TP.
We say that a finite sequence $a_0,a_1,\ldots, a_n$ is PF$_r$ (PF, resp.)
if the corresponding infinite sequence $a_0,a_1,\ldots,a_n,0,\ldots$ is PF$_r$ (PF, resp.).
%The fundamental representation theorem of Schoenberg and Edrei states
%that a sequence $(a_n)_{n\ge 0}$ of real numbers is PF if and only if its generating function has the form
%$$\sum_{n\ge 0}a_nz^n=\frac{\prod_{j\ge 1}(1+\alpha_jz)}{\prod_{j\ge 1}(1-\beta_jz)}e^{\gamma z}$$
%in some open disk centered at the origin,
%where $\alpha_j,\beta_j,\gamma\ge 0$ and $\sum_{j\ge 1}(\alpha_j + \beta_j) < +\infty$
%(see \cite[p. 412]{Kar68} for instance).
%In particular,
A fundamental result of Aissen, Schoenberg and Whitney states that
a finite sequence of nonnegative numbers is PF
if and only if its generating function has only real zeros
(see \cite[p. 399]{Kar68} for instance).
For example, the sequence $(r,s,t)$ of nonnegative numbers is PF if and only if $s^2\ge 4rt$.
We say that a nonnegative sequence $(a_n)$ is {\it log-convex} ({\it log-concave}, resp.)
if $a_{i}a_{j+1}\ge a_{i+1}a_{j}$ ($a_{i}a_{j+1}\le a_{i+1}a_{j}$, resp.) for $0\le i<j$.
Clearly, the sequence $(a_n)$ is log-concave if and only if it is PF$_2$,
i.e., its Toeplitz matrix $[a_{i-j}]_{i,j\ge 0}$ is TP$_2$,
and the sequence is log-convex if and only if its Hankel matrix $[a_{i+j}]_{i,j\ge 0}$ is TP$_2$~\cite{Bre89}.

There are often various total positivity properties in a Riordan array.
%Many of Riordan arrays have various total positivity properties.
For example, the Pascal matrix is TP~\cite[p. 137]{Kar68}
and each row of it is log-concave (see \cite{SW08} for more information),
the Catalan numbers, the Motzkin numbers, the large and little Schr\"oder numbers form a log-convex sequence respectively~\cite{LW07b}.
However, there is no systematic study of total positivity of Riordan arrays.
The object of this paper is to study various positivity properties of Riordan arrays,
including the total positivity of such a matrix,
the log-convexity of the $0$th column and the log-concavity of each row.
The paper is organized as follows.
In the next section,
we present sufficient conditions for total positivity of Riordan arrays.
As applications,
we show that many well-known combinatorial triangles are totally positive
and many famous combinatorial numbers are log-convex in a unified approach.
In Section 3,
we propose some problems for further work.
%%%%%%%%%%%%%%%%%%%%%%%%%%%%%%%%%%%%%%%%%%%
\section{Main results and applications}
\hspace*{\parindent}
%%%%%%%%%%%%%%%%%%%%%%%%%%%%%%%%%%%%%%%%%%%
We first present a basic result about total positivity of Riordan arrays.
Let $R=[r_{n,k}]_{n,k\ge 0}$ be a Riordan array
%with $A=(a_n)_{n\ge 0}$ and $Z=(z_n)_{n\ge 0}$.
defined by the recursive system (\ref{rrr-c}).
%Then $A(R)=(a_n)_{n\ge 0}$ and $Z(R)=(z_n)_{n\ge 0}$.
Call
\begin{equation*}\label{J-R}
J(R)=\left[
\begin{array}{cccccc}
z_0 & a_{0} &  &  &  &\\
z_1 & a_{1} & a_{0} & &\\
z_2 & a_{2} & a_{1} & a_{0} &  &\\
z_3 & a_{3} & a_{2} & a_{1} & a_{0} &\\
\vdots &\vdots & &\cdots & & \ddots\\
\end{array}
\right].
%=[\zeta,A],
\end{equation*}
the {\it coefficient matrix} of the Riordan array $R$.

\begin{thm}\label{grr-thm}
Let $R$ be a Riordan array defined by (\ref{rrr-c}).
%with $A=(a_n)_{n\ge 0}$ and $Z=(z_n)_{n\ge 0}$.
\begin{itemize}
  \item [\rm (i)] If the coefficient matrix $J(R)$ is TP$_r$ (TP, resp.), then so is $R$.
  \item [\rm (ii)] If $R$ is TP$_2$ and all $z_n\ge 0$,
  then the $0$th column $(r_{n,0})_{n\ge 0}$ of $R$ is log-convex.
\end{itemize}
\end{thm}

To prove Theorem~\ref{grr-thm},
we need two lemmas.
The first is direct by definition and
the second follows from the classic Cauchy-Binet formula.

\begin{lem}\label{lps-lem}
A matrix is TP$_r$ (TP, resp.) if and only if
its leading principal submatrices are all TP$_r$ (TP, resp.).
\end{lem}

\begin{lem}\label{prod-lem}
If two matrices are TP$_r$ (TP, resp.),
then so is their product.
\end{lem}

\begin{proof}[Proof of Theorem~\ref{grr-thm}]
(i)\quad
It suffices to show that $J(R)$ is TP$_r$ implies $R$ is TP$_r$.
%Clearly, an infinite matrix is TP$_r$ if and only if its leading principal submatrices are all TP$_r$.
Let
$$R_n=\left[
\begin{array}{ccccc}
r_{0,0} &  &  &  &  \\
r_{1,0} & r_{1,1} &  &  &  \\
r_{2,0} & r_{2,1} & r_{2,2} &  &  \\
\vdots & & & \ddots &  \\
r_{n,0} & r_{n,1} & r_{n,2} & \cdots & r_{n,n}
\end{array}
\right]$$
be the $n$th leading principal submatrix of $R$.
Then by Lemma~\ref{lps-lem},
it suffices to show that $R_n$ is TP$_r$ for $n\ge 1$.
We proceed by induction on $n$.
%Obviously, $R_1$ is TP$_r$.
Assume that $R_n$ is TP$_r$.
By (\ref{rrr-c}), we have
$$\left[
\begin{array}{rrrr}
r_{0,0} & & & \\
r_{1,0} & r_{1,1} &  &\\
\vdots &  &  \ddots  & \\
r_{n+1,0} & r_{n+1,1} & \cdots & r_{n+1,n+1}\\
\end{array}
\right]\\
=
\left[
\begin{array}{rrrr}
1&&&\\
0&r_{0,0} &  &\\
\vdots & \vdots & \ddots  & \\
0&r_{n,0}&\cdots&r_{n,n}\\
\end{array}
\right]
\left[
\begin{array}{rrrr}
1 & & &\\
z_0 & a_{0} &  &\\
\vdots & \vdots &     \ddots  & \\
z_n&a_n&\cdots&a_0\\
\end{array}
\right],
$$
or briefly,
\begin{equation}\label{r=rj}
R_{n+1}=\left[\begin{array}{cc}1 & O\\ O & R_n\\\end{array}\right]
\left[\begin{array}{cc}1 & O\\ \zeta_n & \mathcal{A}_n\\\end{array}\right],
\end{equation}
where $\zeta_n=[z_0,z_1,\ldots,z_n]'$ and $\mathcal{A}_n=[a_{i-j}]_{0\le i,j\le n}$.
By the induction hypothesis, $R_n$ is TP$_r$,
so is the first matrix on the right hand side of (\ref{r=rj}).
On the other hand,
$[\zeta_n,\mathcal{A}_n]$ is TP$_r$ since it is a submatrix of the TP$_r$ matrix $J(R)$,
so is the second matrix on the right hand side of (\ref{r=rj}).
It follows from Lemma~\ref{prod-lem} that the product $R_{n+1}$ is TP$_r$.
Thus the matrix $R$ is TP$_r$.

(ii)\quad
Note that
\begin{eqnarray}\label{r2-dec}
\left[\begin{array}{cc}
r_{0,0} & r_{1,0}\\
r_{1,0} & r_{2,0}\\
r_{2,0} & r_{3,0}\\
\vdots & \vdots\\
\end{array}
\right]=
\left[
\begin{array}{ccccc}
r_{0,0} & & & & \\
r_{1,0} & r_{1,1} &  &  &  \\
r_{2,0} & r_{2,1} & r_{2,2} &  &  \\
\vdots & & & & \ddots   \\
\end{array}
\right]
\left[\begin{array}{cc}
1 & z_{0}\\
0 & z_{1}\\
0 & z_{2}\\
\vdots & \vdots\\
\end{array}
\right].
\end{eqnarray}
Clearly, the second matrix on the right hand side of (\ref{r2-dec}) is TP$_2$ since all $z_n$ are nonnegative.
Now $R$ is TP$_2$ by the assumption.
Hence the matrix on the left hand side of (\ref{r2-dec}) is TP$_2$ by Lemma~\ref{prod-lem}.
In other words, the sequence $(r_{n,0})_{n\ge 0}$ is log-convex.
\end{proof}

In the sequel we consider applications of Theorem~\ref{grr-thm} to two classes of special interesting Riordan arrays.
The first are recursive matrices introduced by Aigner \cite{Aig99,Aig01}.
Let $a,b,s,t$ be four nonnegative numbers.
Define a Riordan array $R(a,b;s,t)=[r_{n,k}]_{n,k\ge 0}$ by
$$r_{0,0}=1,\quad r_{n+1,0}=ar_{n,0}+br_{n,1},\quad r_{n+1,k}=r_{n,k-1}+sr_{n,k}+tr_{n,k+1}.$$
%where $r_{n,k}=0$ unless $0\le k\le n$.
%For such a Riordan array, we have $Z(x)=a+bx$ and $A(x)=1+sx+tx^2$.
Following Aigner~\cite{Aig99,Aig01},
the numbers $C_n(a,b;s,t)=r_{n,0}$ are called the {\it Catalan-like numbers}.
Many well-known triangles are recursive matrices.
For example,
the Pascal triangle,
the Catalan triangle
and the Motzkin triangle
are $R(1,0;1,0),R(2,1;2,1)$ and $R(1,1;1,1)$ respectively.
Also, %as pointed out by Aigner~\cite{Aig99,Aig01},
the Catalan-like numbers unify many famous counting coefficients,
such as the Catalan numbers $C_n(2,1;2,1)$,
the Motzkin numbers $C_n(1,1;1,1)$,
the central binomial coefficients $C_n(2,2;2,1)$,
and the large Schr\"oder numbers $C_n(2,2;3,2)$.
%and the restricted hexagonal numbers $C_n(3,1;3,1)$.
See \cite{Aig01} for details.

\begin{thm}\label{s-t}
Let $a,b,s,t$ be four nonnegative numbers.
\begin{itemize}
  \item [\rm (i)] If $as\ge b$ and $s^2\ge t$, then the sequence $(r_{n,0})_{n\ge 0}$ is log-convex.
  \item [\rm (ii)] If $s^2\ge 4t$ and $a\frac{s+\sqrt{s^2-4t}}{2}\ge b$,
                   then the matrix $R(a,b;s,t)$ is totally positive.
\end{itemize}
\end{thm}

\begin{rem}
From Theorem \ref{s-t} it follows immediately
that the Pascal triangle and the Catalan triangle are totally positive,
and that the Catalan numbers, the Motzkin numbers,
the central binomial coefficients,
and the large Schr\"oder numbers are log-convex respectively.
\end{rem}
\begin{rem}
Let
$$H(a,b;s,t)=\left[r_{n+m,0}\right]_{n,m\ge 0}=
\left[
  \begin{array}{cccc}
    r_{0,0} & r_{1,0} & r_{2,0} & \cdots \\
    r_{1,0} & r_{2,0} & r_{3,0} & \cdots \\
    r_{2,0} & r_{3,0} & r_{4,0} & \cdots \\
    \vdots & \vdots & \vdots & \ddots \\
  \end{array}
\right]$$
be the Hankel matrix of the Catalan-like numbers $r_{n,0}$.
Aigner \cite{Aig99,Aig01} computed the determinants of the leading principal submatrices of $H$.
Aigner's Fundamental Theorem in \cite{Aig01} gives $H=RTR'$,
where $T={\rm diag}(1,t,t^2,t^3,\ldots)$.
So the total positivity of $R$ implies that of the Hankel matrix $H$.
In particular, if $s^2\ge 4t$ and $a\frac{s+\sqrt{s^2-4t}}{2}\ge b$,
then the Hankel matrix $H(a,b;s,t)$ is totally positive.
\end{rem}

By Theorem \ref{grr-thm}, 
to prove Theorem \ref{s-t},
it suffices to prove that
the coefficient matrix of $R(a,b;s,t)$
$$
%J(a,b;s,t)=
\left[
\begin{array}{cccccc}
a & 1 &  &  & & \\
b & s & 1 & &\\
 & t & s & 1 &  &\\
 & & t & s & 1 & \\
& & &\ddots & \ddots  &  \ddots  \\
\end{array}
\right]$$
is TP$_2$ and TP under the conditions respectively.
We do this by establishing the following stronger result.

\begin{prop}
Let $a,b,r,s,t$ be five nonnegative numbers and
the Jacobi matrix
$$J=\left[
\begin{array}{cccccc}
a & r &  &  & & \\
b & s & r & &\\
 & t & s & r &  &\\
 & & t & s & r & \\
& & &\ddots & \ddots  &  \ddots  \\
\end{array}
\right].$$
Then
\begin{itemize}
  \item [\rm (i)] $J$ is TP$_2$ if and only if $as\ge br$ and $s^2\ge rt$.
  \item [\rm (ii)] $J$ is TP if and only if $s^2\ge 4rt$ and $a\left(s+\sqrt{s^2-4rt}\right)/2\ge br$.
\end{itemize}
\end{prop}
\begin{proof}
(i) is obvious, and it remains to prove (ii).

Clearly, the tridiagonal matrix $J$ is TP if and only if
all its principal minors containing consecutive rows and columns are nonnegative
(see, e.g., \cite[Theorem 4.3]{Pin10}),
i.e., all determinants of forms
%So it suffices to show that
$$d_n=\det\left[
\begin{array}{ccccc}
s & r &  &  &  \\
t & s & r & \\
 & t & s & \ddots &  \\
&  &\ddots & \ddots  &  r  \\
& & & t & s \\
\end{array}
\right]_{n\times n}$$
and
$$D_n=\det\left[
\begin{array}{ccccc}
a & r &  &  &  \\
b & s & r & \\
 & t & s & \ddots &  \\
&  &\ddots & \ddots  &  r  \\
& & & t & s \\
\end{array}
\right]_{(n+1)\times (n+1)}$$
are nonnegative.
Note that all $d_n$ are nonnegative if and only if the Toeplitz matrix
$$\left[
\begin{array}{ccccc}
r &  &  & & \\
s & r & &\\
t & s & r &  &\\
& t & s & r & \\
& &\ddots & \ddots  &  \ddots  \\
\end{array}
\right]$$
is TP, i.e., the sequence $(r,s,t)$ is PF.
Hence all $d_n\ge 0$ if and only if $s^2\ge 4rt$.
We complete the proof of (ii) 
by showing that all $D_n\ge 0$ if and only if $(s+\sqrt{s^2-4rt})/2\ge br/a$.

Note that $D_0=a$ and $D_n=ad_{n}-brd_{n-1}$
by expanding the determinant along the first column,
where $d_0=1$ and $d_1=s$.
Hence all $D_n\ge0$ if and only if all $d_n/d_{n-1}\ge br/a$.
We next show that the sequence $d_n/d_{n-1}$ is nonincreasing and convergent to $(s+\sqrt{s^2-4rt})/2$,
which means that all $d_n/d_{n-1}\ge br/a$ is equivalent to $(s+\sqrt{s^2-4rt})/2\ge br/a$,
and so all $D_n\ge 0$ if and only if $(s+\sqrt{s^2-4rt})/2\ge br/a$.

Actually, we have by expanding the determinant along the first column 
\begin{equation}\label{d-r}
d_n=sd_{n-1}-rtd_{n-2},\quad d_0=1,d_1=s.
\end{equation}
Solving this linear recurrence relation we obtain
$$d_n=\sum_{k=0}^n\la^k\mu^{n-k},$$
where
$$\la=\frac{s+\sqrt{s^2-4rt}}{2},\quad \mu=\frac{s-\sqrt{s^2-4rt}}{2}$$
are two roots of the characteristic equation $x^2-sx+rt=0$ of (\ref{d-r}).

By (\ref{d-r}), we have
$$\left[
    \begin{array}{ll}
      d_{n} & d_{n+1} \\
      d_{n-1} & d_{n} \\
    \end{array}
  \right]
=\left[
    \begin{array}{rr}
      s & -rt \\
      1 & 0 \\
    \end{array}
  \right]\left[
    \begin{array}{ll}
      d_{n-1} & d_{n} \\
      d_{n-2} & d_{n-1} \\
    \end{array}
  \right].$$
It follows that $d_{n}^2-d_{n-1}d_{n+1}=rt(d_{n-1}^2-d_{n-2}d_{n})=\cdots=(rt)^{n-1}(d_1^2-d_0d_2)=(rt)^n\ge 0$.
Thus the sequence $d_n/d_{n-1}$ is nonincreasing, and is therefore convergent.
Let $\alpha$ be the limit.
Rewrite (\ref{d-r}) as
$$\frac{d_{n}}{d_{n-1}}=s-\frac{rt}{d_{n-1}/d_{n-2}}.$$
Take the limit to obtain $\alpha=s-rt/\alpha$, i.e., $\alpha^2-s\alpha+rt=0$,
so $\alpha=\la$ or $\mu$.
Now $d_1/d_0=s\ge \la$.
Assume that $d_{n-1}/d_{n-2}\ge \la$.
Then $$\frac{d_{n}}{d_{n-1}}=s-\frac{rt}{d_{n-1}/d_{n-2}}\ge s-\frac{rt}{\la}=\la.$$
Thus all $d_n/d_{n-1}\ge \la$.
It follows that the limit $\alpha=\la$ since $\la\ge \mu$, as desired.
\end{proof}

The second we concern about are Riordan arrays
whose $A$- and $Z$-sequences are identical or nearly so.
%just like the Pascal triangle and the Catalan triangle.
Following \cite{CKS12},
we say that a Riordan array $R=[r_{n,k}]_{n,k\ge 0}$ is {\it consistent} if $A=Z$.
We say that $R$ is a {\it quasi-consistent} Riordan array
if $A=(a_0,a_1,a_2,\ldots)$ and $Z=(a_1,a_2,\ldots)$.
In this case, we have
$$r_{n+1,k}=a_0r_{n,k-1}+a_1r_{n,k}+a_2r_{n,k+1}+\cdots$$
for all $n,k\ge 0$, where $r_{n,j}=0$ unless $0\le j\le n$.
For example,
the little Schr\"oder triangle is consistent,
the Pascal triangle, the Catalan triangle, the Motzkin triangle and the large Schr\"oder triangle are quasi-consistent,
and the ballot table is both consistent and quasi-consistent.
The following theorem gives a unified settle for the total positivity of
these well-known triangles.
%These triangles are all totally positive matrices from the following theorem.
%Applying Theorem~\ref{grr-thm} to (quasi-)consistent Riordan arrays,
%we obtain the following result.

\begin{thm}\label{mainthm}
Let $R$ be a consistent or quasi-consistent Riordan array. %with the $A$-sequence $A=(a_n)_{n\ge 0}$.
Suppose that the $A$-sequence of $R$ is PF$_r$  (PF, resp.).
Then $R$ is TP$_r$ (TP, resp.).
In particular, if the $A$-sequence of $R$ is log-concave,
then the $0$th column $(r_{n,0})_{n\ge 0}$ of $R$ is log-convex,
and each row $(r_{n,k})_{0\le k\le n}$ of $R$ is log-concave.
\end{thm}
\begin{proof}
%We first consider the case that $R$ is consistent Riordan array.
For a consistent or quasi-consistent Riordan array $R$,
its coefficient matrix is
$$
J(R)=
\left[
\begin{array}{ccccc}
a_0 & a_{0} &  &  &  \\
a_1 & a_{1} & a_{0} & \\
a_2 & a_{2} & a_{1} & a_{0} &  \\
\vdots &\vdots  & &   & \ddots   \\
\end{array}
\right]
\quad \textrm{or}\quad
\left[
\begin{array}{ccccc}
a_1 & a_{0} &  &  &  \\
a_2 & a_{1} & a_{0} & \\
a_3 & a_{2} & a_{1} & a_{0} &  \\
\vdots & \vdots & &   & \ddots   \\
\end{array}
\right]
$$
respectively.
Clearly, $J(R)$ is TP$_r$ if and only if
the Toeplitz matrix $[a_{i-j}]_{i,j\ge 0}$ of the sequence $(a_n)_{n\ge 0}$ is TP$_r$,
or equivalently, the sequence $(a_n)_{n\ge 0}$ is PF$_r$.
Thus it follows from Theorem~\ref{grr-thm}~(i) that the Riordan array $R$ is TP$_r$ (resp., TP)
if its $A$-sequence is PF$_r$ (resp. PF).

In particular,
if $(a_n)_{n\ge 0}$ is PF$_2$,
then $R$ is TP$_2$.
It follows from Theorem~\ref{grr-thm}~(ii) that the sequence $(r_{n,0})_{n\ge 0}$ is log-convex.
In what follows we show that each row of $R$ is log-concave by induction.
Denote $s_i=r_{n,i}$ for $0\le i\le n$ and $t_j=r_{n+1,j}$ for $0\le j\le n+1$.
We distinguish two cases.

First consider the case that $R$ is a consistent Riordan array.
In this case, $t_0=t_1$ and
$$t_{k}-t_{k+1}=a_0(s_{k-1}-s_{k})+\cdots+a_{n-k}(s_{n-1}-s_{n})+a_{n-k+1}s_{n}$$
for $1\le k\le n$.
It follows that each row of $R$ is nonincreasing by induction.
On the other hand,
\begin{equation}\label{row}
  \left[
    \begin{array}{llll}
      t_{n+1} &  &  &   \\
      t_{n} & t_{n+1} &  &   \\
       \vdots & \vdots & \ddots & \\
      t_1 & t_2 & \cdots & t_{n+1} \\
    \end{array}
  \right]
= \left[
    \begin{array}{llll}
      s_{n} &  &  &  \\
      s_{n-1} & s_{n} &  & \\
       \vdots & \vdots & \ddots & \\
      s_0 & s_1 & \cdots & s_{n} \\
    \end{array}
  \right]
   \left[
    \begin{array}{llll}
      a_{0} &  &  &  \\
      a_{1} & a_{0} &  &  \\
       \vdots & \vdots & \ddots & \\
      a_n & a_{n-1} & \cdots & a_0 \\
    \end{array}
  \right].
\end{equation}
Suppose that the sequence $(a_n)_{n\ge 0}$ is log-concave.
Then its Toeplitz matrix $\mathcal{A}=[a_{i-j}]_{i,j\ge 0}$ is TP$_2$,
and so are the leading principal submatrices of $\mathcal{A}$.
Thus the second matrix on the right hand side of (\ref{row}) is TP$_2$.
If the $n$th row $s_0,s_1,\ldots, s_n$ of $R$ is log-concave,
then so is the reverse sequence $s_n,\ldots,s_1,s_0$,
which implies that the first matrix on the right hand side of (\ref{row}) is TP$_2$.
Thus the matrix on the left hand side of (\ref{row}) is TP$_2$ by Lemma \ref{prod-lem},
or equivalently,
the sequence $t_{n+1},t_n,\ldots,t_1$ is log-concave,
and so is the reverse sequence $t_1,\ldots,t_n,t_{n+1}$.
Note that $t_0=t_1\ge t_2$.
Hence the sequence $t_0,t_1,t_2,\ldots,t_{n+1}$ is also log-concave.
Thus each row of $R$ is log-concave by induction.

Next let $R$ be a quasi-consistent Riordan array.
Then
$$
\left[
    \begin{array}{llll}
      t_{n+1} &  &  &  \\
      t_{n} & t_{n+1} &  &   \\
       \vdots & \vdots & \ddots & \\
       t_0 & t_1 & \cdots & t_{n+1}\\
    \end{array}
  \right]  \\
  =
\left[
    \begin{array}{lllll}
      s_{n} &  &  & & \\
      s_{n-1} & s_{n} &  & &\\
       \vdots & \vdots & \ddots & &\\
      s_0 & s_1 & \cdots & s_{n}& \\
       0&s_0 & \cdots & s_{n-1} & s_{n} \\
    \end{array}
  \right]
   \left[
    \begin{array}{llll}
      a_{0} &  &  &  \\
      a_{1} & a_{0} &  & \\
       \vdots & \vdots & \ddots &\\
            a_{n+1}&a_n  & \cdots &  a_0 \\
    \end{array}
  \right].
$$
Assume that the sequence $s_0,s_1,\ldots,s_n$ is log-concave.
Then the first matrix on the right hand side is TP$_2$.
It follows that the matrix on the left hand side is TP$_2$.
In other words, the sequence $t_0,t_1,\ldots,t_n,t_{n+1}$ is log-concave.
Thus each row of $R$ is log-concave by induction.
This completes the proof.
\end{proof}
%%%%%%%%%%%%%%%%%%%%%%%%%%%%%%%%%%%%%%%%%%%
\section{Further work}
\hspace*{\parindent}
%%%%%%%%%%%%%%%%%%%%%%%%%%%%%%%%%%%%%%%%%%%
It is known that sequences of binomial coefficients located in a ray or a transversal of the Pascal triangle
have various positivity properties (see \cite{SW08,Yu09} for instance).
Similar problems naturally arise in a Riordan array.
For example, in which case each row of such a Riordan array is PF,
the corresponding linear transformation can preserve the PF property (the log-concavity, the log-convexity, resp.),
and each column of the array is first log-concave and then log-convex?

Aigner \cite{Aig01} gave combinatorial interpretations for recursive matrices in terms of weighted Motzkin paths.
Cheon {\it et al.} \cite{CKS12} provided combinatorial interpretations for consistent Riordan arrays
in terms of weighted {\L}ukasiewicz paths.
It is not difficult to give a similar combinatorial interpretation for a quasi-consistent Riordan array.
Brenti \cite{Bre95} gave combinatorial proofs of total positivity of many well-known matrices
by means of lattice path techniques.
It is natural to ask for combinatorial proofs of Theorems \ref{s-t} and \ref{mainthm}.

\section*{Acknowledgement}

This work was supported in part by the National Natural Science Foundation of China (Grant Nos. 11071030, 11371078)
and the Specialized Research Fund for the Doctoral Program of Higher Education of China (Grant No. 20110041110039).

%\newpage
%\small

\end{document}